\documentclass[11pt]{article}
\usepackage{bbm}
\usepackage[latin1]{inputenc}
\usepackage{amsmath,amsthm,amssymb}
\usepackage{amsfonts}
\usepackage{amsmath,amsthm,amssymb,amscd}
\usepackage{latexsym}
\usepackage{color}
\usepackage{graphicx}
\usepackage{mathrsfs}

\makeatletter \@addtoreset{equation}{section} \makeatother

\newtheorem{theorem}{Theorem}[section]
\newtheorem{definition}{Definition}[section]
\newtheorem{proposition}{Proposition}[section]
\newtheorem{lemma}{Lemma}[section]
\newtheorem{remark}{Remark}[section]

\begin{document}

\title{ Degenerate fractional Kirchhoff-type system with \\
magnetic 
 fields and upper critical growth}

\author{ {Mingzhe Sun$^{\small\mbox{a,b}}$,  Shaoyun
Shi$^{\small\mbox{a}}$ and Du\v{s}an D. Repov\v{s}$^{\small\mbox{c,d,e}}$\thanks{{ {{\it E-mail address:}  mathmzhsun@163.com (M. Sun), shisy@jlu.edu.cn (S. Shi), dusan.repovs@guest.arnes.si (D.D. Repov\v{s})} }}
\thanks{{ { Corresponding author: Du\v{s}an D. Repov\v{s} } }}}\\
$^{\small\mbox{a}}${\small  School of Mathematics, Jilin University,
130012
Changchun, PR China}\\
$^{\small\mbox{b}}${\small  Department of Mathematics, Yanbian
University,  133002 Yanji, PR China }\\
$^{\small\mbox{c}}${\small Faculty of Education, University of Ljubljana,  1000 Ljubljana, 
Slovenia }\\
$^{\small\mbox{d}}${\small  Faculty of
Mathematics and Physics, University of Ljubljana,  1000 Ljubljana, 
Slovenia }\\
$^{\small\mbox{e}}${\small Institute of
Mathematics, Physics and Mechanics,  1000 Ljubljana, 
Slovenia }
}
\date{}
\maketitle

\maketitle
\begin{abstract}
This paper deals with the following degenerate fractional
Kirchhoff-type system with magnetic fields and critical growth:
$$
\left\{
\begin{array}{lll}
 -\mathfrak{M}(\|u\|_{s,A}^2)[(-\Delta)^s_Au+u] =
G_u(|x|,|u|^2,|v|^2) +
\left(\mathcal{I}_\mu*|u|^{p^*}\right)|u|^{p^*-2}u  \ &\mbox{in}\,\,\mathbb{R}^N,\\
\mathfrak{M}(\|v\|_{s,A})[(-\Delta)^s_Av+v] = G_v(|x|,|u|^2,|v|^2) +
\left(\mathcal{I}_\mu*|v|^{p^*}\right)|v|^{p^*-2}v \
&\mbox{in}\,\,\mathbb{R}^N,
\end{array}\right.
$$
where
$$\|u\|_{s,A}=\left(\iint_{\mathbb{R}^{2N}}\frac{|u(x)-e^{i(x-y)\cdot
A(\frac{x+y}{2})}u(y)|^2}{|x-y|^{N+2s}}dx
dy+\int_{\mathbb{R}^N}|u|^2dx\right)^{1/2},$$  and $(-\Delta)_{A}^s$
and $A$ are called magnetic operator and magnetic potential,
respectively. $\mathfrak{M}:\mathbb{R}^{+}_{0}\rightarrow
\mathbb{R}^{+}_0$ is a continuous Kirchhoff function,
$\mathcal{I}_\mu(x) = |x|^{N-\mu}$ with $0<\mu<N$, $C^1$-function
$G$ satisfies some suitable conditions, and $p^* =\frac{N+\mu}{N-2s}$. We
prove the multiplicity results for this problem using the limit
index theory. The novelty of our work is the appearance of  
convolution terms and critical nonlinearities. To overcome the
difficulty caused by degenerate Kirchhoff function and
critical nonlinearity, we introduce several analytical tools and
the fractional version concentration-compactness principles which
are useful tools for proving the compactness condition.
\end{abstract}

{\sl Keywords and phrases:}\,\, Fractional Kirchhoff-type system; Upper
critical exponent; Concentration-compactness principle; Variational
method; Multiple solutions.

\emph{\sl Math. Subj. Classif. (2020):}  35J20, 35J60, 35R11.

\section{Introduction}\label{s1}

This paper deals with the following degenerate fractional
Kirchhoff-type system with magnetic fields and critical growth:
\begin{eqnarray}\label{e1.1}
\left\{
\begin{array}{lll}
 -\mathfrak{M}(\|u\|_{s,A}^2)[(-\Delta)^s_Au+u] =
 G_u(|x|,|u|^2,|v|^2) +
\left(\mathcal{I}_\mu*|u|^{p^*}\right)|u|^{p^*-2}u   \ &\mbox{in}\,\,\mathbb{R}^N,\\
\mathfrak{M}(\|v\|_{s,A}^2)[(-\Delta)^s_Av+v] = G_v(|x|,|u|^2,|v|^2)
+ \left(\mathcal{I}_\mu*|v|^{p^*}\right)|v|^{p^*-2}v \
&\mbox{in}\,\,\mathbb{R}^N,
\end{array}\right.
\end{eqnarray}
where
 $$\mathcal{I}_\mu(x)=|x|^{N-\mu},
 \
 \hbox{ with}
 \
 0<\mu<N,
 \quad
 \|z\|_{s}=\left([z]_{s,A}^2
+\int_{\mathbb{R}^N}|z|^2dx\right)^{1/2},$$
$$[z]_{s,A}=\left(\iint_{\mathbb{R}^{2N}}\frac{|z(x)-e^{i(x-y)\cdot
A(\frac{x+y}{2})}z(y)|^2}{|x-y|^{N+2s}}dx dy\right)^{{1}/{2}},$$
$(-\Delta)_{A}^s$ and $A$ are called magnetic operator and
magnetic potential, respectively. According to the
Hardy-Littlewood-Sobolev inequality (see \eqref{e2.2}), the exponent
$p^* =\frac{N+\mu}{N-2s}$ is called  upper critical. The
continuous Kirchhoff function
$\mathfrak{M}:\mathbb{R}^{+}_{0}\rightarrow \mathbb{R}^{+}_0$ and
$C^1$-function $G: [0, +\infty)\times
\mathbb{R}^2\rightarrow\mathbb{R}^+$ will satisfy the following
assumptions throughout the paper:
\begin{itemize}
\item[($\mathcal {M}$)]
$(M_1)$  $\inf_{t>0}\mathfrak{M}(t)=\mathfrak{m}^\ast > 0$.

$(M_2)$  For all $t \in [0, +\infty)$, there exists $\sigma \in (1,
p^\ast/2)$ such that $\sigma\mathscr{M}(t)\geq \mathfrak{M}(t)t$,
where $$\mathscr{M}(t)=\int_0^t\mathfrak{M}(s)ds.$$

$(M_3)$ For all
$t \in (0, +\infty)$, there exists $\mathfrak{m}_1 >
0$ such that $\mathfrak{M}(t) \geq \mathfrak{m}_1 t^{\sigma-1}$,
moreover $\mathfrak{M}(0) = 0$.

\item[($\mathcal {G}$)]
$(G_1)$  For all $(r,\xi,\eta) \in [0, +\infty)\times \mathbb{R}^2$,
there exist  $C>0$ and $2 < \tau < 2^\ast:=\frac{2N}{N-2}$ such that
\begin{displaymath}
|G_\xi(r,\xi,\eta)| + |G_\eta(r,\xi,\eta)| \leq
C\left(|\xi|^{\frac{\tau-1}{2}}+
 |\eta|^{\frac{\tau-1}{2}}\right).
\end{displaymath}

$(G_2)$ For all $(r,\xi,\eta) \in [0, +\infty)\times \mathbb{R}^2$,
there exists $2\sigma < \theta < 2p^\ast$ such that
$$0 < \theta G(r,\xi,\eta) \leq \xi G_\xi(r,\xi,\eta)+\eta G_\eta(r,\xi,\eta),$$ where $\sigma$ is defined by $(M_2)$.

$(G_3)$ $\xi G_\xi(r, \xi,\eta) \geq 0$   $ \ \hbox{for all} \ \ (r,
\xi,\eta) \in [0, +\infty)\times \mathbb{R}^2$.

$(G_4)$ $G(r, \xi,\eta) = G(r, -\xi, -\eta)$ $ \ \hbox{for all} \ \
r \geq 0$ and $\xi, \eta \in \mathbb{R}$.
\end{itemize}
\begin{remark}\label{rem1.1}
A typical function which satisfies conditions $(M_1)$-$(M_3)$ is
given by $\mathfrak{M}(t)=a+b\,t^{\sigma-1}$ for
$t\in\mathbb{R}^+_0$, where $a\in\mathbb{R}^+_0$,
$b\in\mathbb{R}^+_0,$ and $a+b>0$. In particular, when
$\mathfrak{M}(t) \geq d > 0$ for some $d$ and all $t \geq 0$, this
case is said to be non-degenerate, while it is called degenerate if
$\mathfrak{M}(0) = 0$ and $\mathfrak{M}(t)> 0$ for $t > 0$. However,
in proving the compactness condition, the two cases of degenerate
and non-degenerate are completely different, and it is more
complicated in the degenerate case. In this paper, we mainly deal with
the degenerate fractional Kirchhoff-type system with magnetic
fields. Therefore we need to develop new techniques to conquer some
difficulties induced by
the
 degeneration.
\end{remark}

The fractional magnetic operator $(-\Delta)_{A}^s$ was recently
introduced by d'Avenia and Squassina~\cite{da}, which up to
normalization constants, can be defined on smooth functions $u$ as
follows
\begin{eqnarray*}
(-\Delta)_{A}^s u(x) := 2\lim_{\varepsilon \rightarrow 0}
\int_{\mathbb{R}^N \setminus
B_\varepsilon(x)}\frac{u(x)-e^{i(x-y)\cdot
A(\frac{x+y}{2})}u(y)}{|x-y|^{N+2s}}dy, \quad x\in  \mathbb{R}^N.
\end{eqnarray*}
The equation with fractional magnetic operator often arises as a
model for various physical phenomena, in particular in the study of
the infinitesimal generators of L\'{e}vy stable diffusion processes
\cite{EGE}. Vast literature on nonlocal
operators and on their applications exists, we
refer the interested reader to \cite{am, du,  to,x3, x1}. In order to
further research this
type of question by variational methods,  many
scholars have established the basic properties of fractional Sobolev
spaces - for this the reader is referred to \cite{EGE, MRS, pa}.

First,  we make a quick overview of the literature on the magnetic
Schr\"{o}dinger equation. To begin, we note that there are some
works concerning the magnetic
Schr\"{o}dinger equation

\begin{equation}\label{e1.2}
-(\nabla u-{\rm i} A)^2u + V(x)u = f(x, |u|)u,
\end{equation}
which have appeared in recent years,
where the  magnetic operator in \eqref{e1.2} is given by
\begin{displaymath}
-(\nabla u-{\rm i} A)^2u = -\Delta u +2iA(x)\cdot\nabla u +
|A(x)|^2u + iu \mbox{div} A(x).
\end{displaymath}
As stated in Squassina and Volzone \cite{sq1}, up to correcting the
operator by the factor $(1-s)$, it follows that $(-\Delta)^s_A u$
converges to $-(\nabla u-{\rm i} A)^2u$ as $s\rightarrow1$. Thus, up
to normalization, the nonlocal case can be seen as an approximation
of the local one.
Ji and R\u adulescu  \cite{j1} obtained
the multiplicity and concentration properties of solutions for a
class of nonlinear magnetic Schr\"{o}dinger equation by using
variational methods, penalization techniques, and the
Ljusternik-Schnirelmann theory. For more interesting results, we
refer to \cite{j2, liu, xia, zhang3}.
Recently, many researchers
have paid attention to the
 equations with  fractional magnetic operator. In
particular,  Mingqi et al.  \cite{MPSZ}  proved some existence
results 
for Schr$\ddot{\mbox{o}}$dinger--Kirchhoff type equation
involving the fractional $p$--Laplacian and the magnetic operator
\begin{equation}\label{e1.3}
M([u]_{s,A}^2)(-\Delta)_A^su+V(x)u=f(x,|u|)u\quad \text{in
$\mathbb{R}^N$},
\end{equation}
where $f$ satisfies the subcritical growth condition. For the critical
growth case,  Wang and Xiang \cite{WX} have obtained the existence of
two solutions and infinitely many solutions to fractional
Schr$\ddot{\mbox{o}}$dinger-Choquard-Kirchhoff type equations with
external magnetic operator. Subsequently, Liang  et al.
\cite{liang3} investigated the  existence and multiplicity of
solutions to problem \eqref{e1.1} without Choquard--type term in the
non--degenerate case. We draw the attention of the reader to the
degenerate case involving the magnetic operator in Liang et al.
\cite{liang4} and Mingqi et al. \cite{MPSZ}.

For the case $A \equiv 0$ in problem \eqref{e1.1}, there exist
numerous articles  dedicated to the study of the following
Choquard equation,
\begin{align}\label{e1.4}
-\Delta u+V(x)u=(|x|^{-\mu}*F(u))f(u),\quad \ x \in \mathbb{R}^N.
\end{align}
Eq. \eqref{e1.4} can be used to describe many physical models. For
example, it was proposed by Laskin \cite{la} as a result of
expanding the Feynman path integral from the Brownian-like to the
L\'{e}vy-like quantum mechanical paths. The study of existence and
uniqueness of positive solutions to Choquard type equations
attracted a lot of attention of researchers due to its
applications in physical models Pekar \cite{pek}. In d'Avenia et al.
\cite{da1}, the authors obtained the existence of ground state
solutions for  following fractional Choquard equation of the form
\begin{equation}\label{e1.5}
(-\Delta)^su+ \omega u =
\left(\mathcal{K}_\mu*|u|^{p}\right)|u|^{p-2}u, \quad u \in
H^s(\mathbb{R}^N), \quad N\geq 3,
\end{equation}
where $s\in(0, 1)$, $\omega > 0$ is a given parameter, $\mu \in (0,
N)$, and $\mathcal{K}_\mu(x)=|x|^{N-\mu}$ is the Riesz potential. In
Pucci et al. \cite{pu}, the authors obtained the existence of
nonnegative solutions to a class of
Schr\"{o}dinger-Choquard-Kirchhoff-type fractional equation by using
the Mountain pass theorem and the Ekeland variational principle. For
more results for problems with Hardy-Littlewood-Sobolev critical
nonlinearity without the magnetic operator case, see  Cassani and
Zhang
 \cite{cas},  Ma
and Zhang \cite{ma}, and
Song and Shi \cite{so1, so2}.

Once we turn our attention to the critical nonlocal system with
critical nonlinearity, we immediately see that the literature is
relatively scarce. In this case, we can cite  recent works
\cite{fis2, fu, xiang}. We call attention to Furtado et al.
\cite{fu}
who
 dealt with  the following
non--degenerate Kirchhoff system
\begin{eqnarray}\label{e1.6}
\left\{
\begin{array}{lll}
 -m\left(\int_\Omega|\nabla u|^2\right)\Delta u  =
 F_u(x,u,v) +
\mu_1|u|^4u,   \ &\mbox{in}\,\,\Omega,\\
-l\left(\int_\Omega|\nabla v|^2\right)\Delta v  =
 F_v(x,u,v) +
\mu_2|v|^4v,   \ &\mbox{in}\,\,\Omega,
\end{array}\right.
\end{eqnarray}
where $\Omega \subset \mathbb{R}^3$ is a smooth bounded domain, the
nonlinearity $F$ is subcritical and locally superlinear at infinity,
and they obtained multiple solutions with the aid of the symmetric
mountain-pass theorem. For the degenerate case,  Xiang et al.
\cite{xiang} investigated the existence and asymptotic behaviour of
solutions to critical Schr\"{o}dinger-Kirchhoff type systems by
applying the mountain-pass theorem and Ekeland's variational
principle. 

It is well-known that the Limit Index Theory due to Li
\cite{Li1} is one of the most effective methods to study the
existence of infinitely many solutions for the noncooperative
system. For example, Song and Shi \cite{so3} considered the
noncooperative critical nonlocal system with variable exponents,
Fang and Zhang \cite{fa} studied systems of $p\&q$-Laplacian elliptic
equations with critical Sobolev exponent, Liang et al. \cite{liang0}
dealt with a class of noncooperative Kirchhoff-type system involving
the fractional $p$-Laplacian and critical exponents. We also refer
the interested reader to Huang and Li \cite{hu}, Liang and Shi
\cite{liang1}, and Liang and Zhang \cite{liang2} for some applications
of this method. However, to the best of our knowledge, none of the
cited works address the system with upper critical exponent and
magnetic operator in $\mathbb{R}^N$ in the degenerate case.

Inspired by the previously mentioned works, our main objective is to
study the existence and multiple solutions to problem \eqref{e1.1}, by means of  the Limit Index Theory.
To the best of our knowledge, this is the first time in the
literature to use the Limit Index Theory to investigate the
degenerate fractional Kirchhoff-type system with magnetic fields and
upper critical growth.
Our main result is the following.
\begin{theorem}\label{the1.1}
Suppose that assumptions ($\mathcal {M}$) and ($\mathcal {G}$)
are fulfilled. Then problem \eqref{e1.1} has infinitely many
solutions.
\end{theorem}

\begin{remark}\label{rem1.2}
The main difficulty of this paper lies in
the following three aspects: First, in
order to recover the compactness of the Palais-Smale sequence, we
shall use the second concentration-compactness
principle for the convolution type nonlocal problem in the
fractional Sobolev space. Second, the appearance of the magnetic
field also brings additional difficulties to the study of problem
\eqref{e1.1}, such as effects of the magnetic fields on the linear
spectral sets and on the solution structure. Finally, since we
consider problem \eqref{e1.1} in the whole space, in order to apply
the Limit Index Theory, we need to establish new techniques to
overcome this difficulty. To 
the best of our knowledge, our theorem is also
valid for $s = 1$,  hence the corresponding result in this case is
new as well.
\end{remark}

The organization of this paper is as follows:  In Section \ref{s2},
 we give some
basic definitions of fractional Sobolev space and
the
 well known
Hardy-Littlewood-Sobolev inequality. In Section \ref{s3}, we
mainly
introduce the Limit Index Theory. In Section \ref{s4}, we
prove some
compactness lemmas for the functional associated to our problem. The
proof of the main result Theorem \ref{the1.1} is given in Section \ref{s5}.

\section{Preliminaries}\label{s2}

In this section, we collect some known results for the readers
convenience and the  later use. First, we shall give some useful facts
for the fractional order Sobolev spaces. Let $H^{s}(\mathbb{R}^N)$
be a fractional order Sobolev spaces which is defined as follows
\begin{displaymath}
H^{s}(\mathbb{R}^N)  :=  \left\{u \in L^2(\mathbb{R}^N): [u]_{s} <
\infty\right\},
\end{displaymath}
where $[u]_s$ denotes the Gagliardo semi-norm
\begin{displaymath}
[u]_{s}  := \left(\iint_{\mathbb{R}^{2N}}\frac{|u(x)-
u(y)|^2}{|x-y|^{N+2s}}dx dy\right)^{1/2},
\end{displaymath}
equipped with the inner product
\begin{displaymath}
\langle u, v\rangle  :=  \iint_{\mathbb{R}^{2N}}\frac{(u(x)-
u(y))(v(x)- v(y))}{|x-y|^{N+2s}}dx dy + \int_{\mathbb{R}^N}\xi\eta
dx \ \ \hbox{for all} \ \ u, v \in H^{s}(\mathbb{R}^N)
\end{displaymath}
and the norm
\begin{displaymath}
\|u\|_s := \left([u]_{s}^2 +
\int_{\mathbb{R}^N}|u|^2dx\right)^{\frac{1}{2}}.
\end{displaymath}
Here, $L^2(\mathbb{R}^N)$ denotes the Lebesgue space of real-valued
functions with $\int_{\mathbb{R}^N}|u|^2dx < \infty$. From
 \cite[Theorem 6.7 ]{EGE}, we know that the embedding
$H^{s}(\mathbb{R}^N) \hookrightarrow L^t(\mathbb{R}^N)$ is
continuous for any $t \in [2, 2_s^\ast]$. Moreover, there exists a
positive constant $C_t$ such that
\begin{equation}\label{e2.2}
|u|_{L^t(\mathbb{R}^N)} \leq C_t\|u\|_s \quad \mbox{for\ all}\ u \in
H^{s}(\mathbb{R}^N).
\end{equation}
In order to obtain the existence of radial weak solutions to system
\eqref{e1.1}, we shall use the following  embedding theorem due to
Lions \cite{lions}.
\begin{theorem}\label{the2.1}
Assume that $0 < s < 1$ and $2s < N$. Then  the embedding
\begin{displaymath}
H_r^{s}(\mathbb{R}^N) \hookrightarrow\hookrightarrow
L^t(\mathbb{R}^N),
\end{displaymath}
is   compact for any $2 < t< 2_s^\ast$, where
$H_r^{s}(\mathbb{R}^N)$ is radial symmetric space, defined
by
\begin{displaymath}
H_r^{s}(\mathbb{R}^N) := \{u \in H^{s}(\mathbb{R}^N): u(x) = u(|x|),
x\in \mathbb{R}^N\}.
\end{displaymath}
\end{theorem}
Suppose that $A : \mathbb{R}^N \rightarrow \mathbb{R}^N$ is a
continuous function.
Then
\begin{displaymath}
[u]_{s,A}  = \left(\iint_{\mathbb{R}^{2N}}\frac{|u(x)-e^{i(x-y)\cdot
A(\frac{x+y}{2})}u(y)|^2}{|x-y|^{N+2s}}dx dy\right)^{1/2}
\end{displaymath}
is the Gagliardo semi-norm. Define
\begin{displaymath}
\mathcal {H}^{s}   :=  \left\{u \in L^2(\mathbb{R}^N, \mathbb{C}):
[u]_{s,A} < \infty\right\}.
\end{displaymath}
It can be endowed with the norm
\begin{displaymath}
\|u\|_{s,A} := \left([u]_{s,A}^2 +
\int_{\mathbb{R}^N}|u|^2dx\right)^{\frac{1}{2}}.
\end{displaymath}
The scalar product on $\mathcal {H}^{s}$ is defined by
\begin{displaymath}
(\xi, \eta)_{s,A} := \langle \xi, \eta\rangle_{L^2} + \langle\xi,
\eta\rangle_{s,A},
\end{displaymath}
where
$$\langle \xi, \eta\rangle_{s,A} =
\mathscr{R}\iint_{\mathbb{R}^{2N}}\frac{(\xi(x)-e^{i(x-y)\cdot
A(\frac{x+y}{2})}\xi(y))\overline{(\eta(x)-e^{i(x-y)\cdot
A(\frac{x+y}{2})}\eta(y))}}{|x-y|^{N+2s}}dx dy.$$ 
By
  \cite[Proposition
2.1]{da}, one knows that $(\mathcal
{H}^{s}, (\cdot,\cdot)_{s,A})$ is a real Hilbert space. Moreover,
the space $C_c^\infty(\mathbb{R}^N, \mathbb{C})$ is a subspace of
$\mathcal {H}^{s}$, see   
\cite[Proposition 2.2]{da}.

Let $H_A^s(\mathbb{R}^N)$ 
be
the closure of
$C_c^\infty(\mathbb{R}^N, \mathbb{C})$ in $\mathcal {H}^{s}$. Then
we have the following lemma, the proof of which
 can be found in
d'Avenia and  Squassina \cite{da}.
\begin{lemma}\label{lem2.1}
Let $u\in H_A^s(\mathbb{R}^N).$
Then $|u|\in H^s(\mathbb{R}^N)$, i.e.,
\begin{displaymath}
\||u|\|_{s}\leq \|u\|_{s,A} \quad \text{for all }
 u\in H^s_{A}(\mathbb{R}^N).
\end{displaymath}
\end{lemma}

Following the same discussion as in d'Avenia and  Squassina \cite{da}, 
together with Lemma \ref{lem2.1}, we arrive at the following embedding
result.
\begin{lemma}\label{lem2.2}
The space $H_A^s(\mathbb{R}^N, \mathbb{C})$ is continuously embedded
in $L^\vartheta(\mathbb{R}^N, \mathbb{C})$  $ \ \hbox{for all} \ $
$\vartheta \in [2, 2_s^\ast]$. Furthermore, the space
$H_A^s(\mathbb{R}^N, \mathbb{C})$ is continuously compactly embedded
in $L^\vartheta(K, \mathbb{C})$ $ \ \hbox{for all} \ $ $\vartheta
\in [2, 2_s^\ast]$  and any compact set $K \subset \mathbb{R}^N$.
\end{lemma}
From Lemma \ref{lem2.1}, Theorem \ref{the2.1}, and  the
Br\'{e}zis-Lieb Lemma, we obtain the following lemma.
\begin{lemma}\label{lem2.3}
Let
\begin{displaymath} H_{r,A}^{s}(\mathbb{R}^N,\mathbb{C}) := \{u \in
H_A^{s}(\mathbb{R}^N,\mathbb{C}): u(x) = u(|x|), x\in
\mathbb{R}^N\}.
\end{displaymath}
Then  the space $H_{r,A}^{s}(\mathbb{R}^N,\mathbb{C})$ is
continuously compactly embedded in $L^\tau(\mathbb{R}^N, \mathbb{C})$
for any $\tau \in (2, 2_s^\ast)$.
\end{lemma}

\indent By \cite[Proposition 3.6 ]{EGE},   we have
\begin{displaymath}
[u]_{s} = \|(-\Delta)^{\frac{s}{2}}\|_{L^2(\mathbb{R}^N)}\quad
\mbox{for\ any}\ u \in H^{s}(\mathbb{R}^N),
\end{displaymath}
i.e.
\begin{displaymath}
\iint_{\mathbb{R}^{2N}}\frac{|u(x)-u(y)|^2}{|x-y|^{N+2s}}dxdy  =
\int_{\mathbb{R}^{N}}|(-\Delta)^{\frac{s}{2}}u(x)|^2dx.
\end{displaymath}
Moreover,
\begin{displaymath}
\iint_{\mathbb{R}^{2N}}\frac{(u(x)-u(y))(v(x)-v(y))}{|x-y|^{N+2s}}dx
dy  =
\int_{\mathbb{R}^{N}}(-\Delta)^{\frac{s}{2}}u(x)\cdot(-\Delta)^{\frac{s}{2}}v(x)
dx.
\end{displaymath}

Next, we recall the well known Hardy--Littlewood--Sobolev inequality,
see   \cite[Theorem~4.3]{lie}.
\begin{lemma}\label{lem2.4}
Assume that  $p,r>1$ and $0<\mu <N$ with $1/p+(N-\mu)/N+1/r=2$,
$f\in L^p(\mathbb{R}^N),$ and $h\in L^r(\mathbb{R}^N)$. Then there exists
a sharp constant $C(p,r,\mu,N),$ independent of $f,h$, such that
\begin{equation}\label{e2.2}
\int_{\mathbb{R}^N}\int_{\mathbb{R}^N}\frac{f(x)h(y)}{|x-y|^{N-\mu}}~dxdy
\leq C(p,r,\mu,N) \|f\|_{L^p} \|h\|_{L^r}.
\end{equation}
Set $p=r=2N/(N+\mu)$. Then
\begin{align*}
C(p,r,\mu,N)=C(N,\mu)=
\pi^{\frac{N-\mu}{2}}\frac{\Gamma(\frac{\mu}{2})}{\Gamma(\frac{N+\mu}{2})}\left\lbrace
\frac{\Gamma(\frac{N}{2})}{\Gamma(N)}\right\rbrace^{\frac{\mu}{N}}.
\end{align*}
\end{lemma}
If $u = v = |w|^q$,  then Lemma \ref{lem2.4} implies that
\begin{displaymath}
\int_{\mathbb R^{N}}\left(\mathcal{I}_\mu*|w|^{q}\right)|w|^{q} \,dx
\end{displaymath}
is well-defined, if $w\in L^{rq}(\mathbb R^{N})$ for some $r > 1$
satisfying $2/r + (N-\mu)/N = 2$. Thus, if $w \in H^s(\mathbb
R^{N})$, then by the Sobolev embedding theorem, 
 we get that $q \in
[p_\ast, p^\ast]$. In particular, in the upper critical case,
\begin{equation}\label{e2.3}
\int_{\mathbb
R^{N}}\left(\mathcal{I}_\mu*|u|^{p^\ast}\right)|u|^{p^\ast} \,dx
\leq C(N,\mu) \|u\|_{2_s^\ast}^{2p^\ast}
\end{equation}
and the equality holds if and only if
\begin{equation}\label{e2.4}
  u=C\left(\frac{l}{l^2+|x-m|^2}\right)^{\frac{N-2}{2}},
\end{equation}
for some $x_0 \in  \mathbb R^{N}$, where $C > 0$ and $l > 0$, see
\cite{lie}. Let
\begin{equation}\label{e2.5}
 S=\inf_{ u \in D^{s}(\mathbb{R}^N) \setminus \{0\}} \left\{
\iint_{\mathbb{R}^{2N}}\frac{|u(x)-u(y)|^2}{|x-y|^{N+2s}}dxdy:\;
\int_{\mathbb{R}^N}|u|^{2_s^{*}}dx=1\right\}
\end{equation}
and
\begin{equation}\label{e2.6}
S_{H} = \inf_{ u \in D^{s}(\mathbb{R}^N)\setminus \{0\}} \left\{
\iint_{\mathbb{R}^{2N}}\frac{|u(x)-u(y)|^2}{|x-y|^{N+2s}}dxdy:\;
\int_{\mathbb{R}^N}
\left(\mathcal{I}_\mu*|u|^{p^\ast}\right)|u|^{p^\ast}~dx=1 \right\}.
\end{equation}
By \eqref{e2.5} and \eqref{e2.3}, $S_{H}$ is achieved if and only
if $u$ satisfies \eqref{e2.4} and $S_{H} =
S/C(N,\mu)^{\frac{1}{p^\ast}},$ see Mukherjee and Sreenadh
\cite{mu1}.

\section{ Limit Index Theory}\label{s3}

\indent In this section,  we shall show that all hypotheses of the
Limit Index Theory  are satisfied and this will eventually yield the
conclusion that there exist infinitely many solutions.  To this end, we introduce some definitions from the  Limit Index
Theory that can be found in Li \cite{Li1} and Willem \cite{w1}, the
reader may also refer to Fang and Zhang \cite{fa}
and Liang and Shi
\cite{liang1}.
\begin{definition}(see \cite{Li1,w1})\label{de2.1}
The action of a topological group $G$ on a normed space $Z$ is a
continuous map
\begin{displaymath}
G \times Z \rightarrow Z :  [g, z] \mapsto gz
\end{displaymath}
such that
\begin{displaymath}
1 \cdot z = z,\quad (gh)z = g(hz)\quad z \mapsto gz\ \mbox{ is\
linear},\ \  \mbox{for\ all}\ g, h \in G.
\end{displaymath}
The action is isometric if
\begin{displaymath}
\|gz\| = \|z\|  \quad\mbox{for\ all}\ g\in G,\ \ z \in Z
\end{displaymath}
and in this case, $Z$ is called a $G$-space.\\
The set of invariant points is defined by
\begin{displaymath}
\mathrm{Fix}(G) := \left\{z \in Z : gz = z, \  \ \hbox{for all} \ \ g \in
G\right\}.
\end{displaymath}
A set $A \subset Z$ is invariant if $gA = A$ for every $g \in G$. A
function $\varphi : Z \rightarrow R$ is invariant $\varphi \circ g =
\varphi$ for every $g \in G$, $z \in Z$. A map $f : Z \rightarrow Z$
is equivariant if $g\circ f = f\circ g$ for every $g \in G$.\\
\indent Assume that $Z$ is a $G$-Banach space, that is, there is a
$G$ isometric action on $Z$. Let
\begin{displaymath}
\Sigma := \left\{A \subset Z: A\ \mbox{is\ closed\ and}\ gA = A,
 \ \hbox{for all} \ \ g \in G\right\}
\end{displaymath}
be a family of all $G$-invariant closed subsets of $Z$, and let
\begin{displaymath}
\Gamma := \left\{h \in C^0(Z, Z):  h(gu) = g(hu), \quad  \ \hbox{for all} \ \ g
\in G\right\}
\end{displaymath}
be the class of all $G$-equivariant mappings of $Z$. Finally, we
call the set
\begin{displaymath}
O(u) := \left\{gu:   g \in G\right\}
\end{displaymath}
the $G$-orbit of $u$.
\end{definition}
\begin{definition}(see \cite{Li1})\label{de2.2}
An index for $(G, \Sigma, \Gamma)$ is a mapping $i: \Sigma
\rightarrow \mathcal {Z_+}\cup \{+\infty\}$ (where $\mathcal {Z_+}$
is the set of all  nonnegative integers) such that for all $A, B \in
\Sigma$, $h \in \Gamma$, the following conditions are satisfied:\\
$\mathrm{(1)}$ $i(A) = 0 \Leftrightarrow A = \emptyset$;\\
$\mathrm{(2)}$ (Monotonicity) $A \subset B \Rightarrow i(A) \leq
i(B)$;\\
$\mathrm{(3)}$ (Subadditivity) $i(A\cup B) \leq i(A) + i(B)$;\\
$\mathrm{(4)}$ (Supervariance) $i(A) \leq i(\overline{h(A)}),
 \ \hbox{for all} \ \ h \in \Gamma$;\\
$\mathrm{(5)}$ (Continuity) If $A$ is compact and
$A\cap\mbox{Fix}(G) = \emptyset$, then $i(A) < +\infty$ and there is
a $G$-invariant
neighbourhood $N$ of $A$ such that $i(\overline{N}) = i(A)$;\\
$\mathrm{(6)}$ (Normalization) If $x \not\in \mbox{Fix}(G)$, then
$i(O(x)) = 1$.
\end{definition}

\begin{definition}(see \cite{benci}) \label{de2.3}
An index theory is said to satisfy the $d$-dimensional property if
there is a positive integer $d$ such that
\begin{displaymath}
i(V^{dk}\cap S_1) = k
\end{displaymath}
for all $dk$-dimensional subspaces $V^{dk} \in \Sigma$ such that
$V^{dk} \cap \mbox{Fix}(G) = \{0\}$, where $S_1$ is the unit sphere
in $Z$.
\end{definition}
\indent Suppose that $U$ and $V$ are $G$-invariant closed subspaces
of $Z$ such that
\begin{displaymath}
Z = U \oplus V,
\end{displaymath}
where $V$ is infinite-dimensional and
\begin{displaymath}
V = \overline{\bigcup_{j=1}^\infty V_j},
\end{displaymath}
where $V_j$ is a $dn_j$-dimensional $G$-invariant subspace of $V$,
$j = 1,2,\cdots,$ and $V_1 \subset V_2 \subset \cdots \subset V_n
\subset \cdots$. Let
\begin{displaymath}
Z_j = U \bigoplus V_j,
\end{displaymath}
and $ \ \hbox{for all} \ \ A \in \Sigma$, let
\begin{displaymath}
A_j = A \bigoplus Z_j.
\end{displaymath}
\begin{definition}(see \cite{Li1})\label{de2.4}
Let $i$ be an index theory satisfying the $d$-dimensional property.
A limit index with respect to $(Z_j)$ induced by $i$ is a mapping
\begin{displaymath}
i^\infty : \Sigma \rightarrow \mathcal {Z} \cup \{-\infty, +\infty\}
\end{displaymath}
given by
\begin{displaymath}
i^\infty (A) = \limsup_{j \rightarrow \infty}(i(A_j)-n_j).
\end{displaymath}
\end{definition}

\begin{proposition}(see \cite{Li1})\label{pro2.1}
Let $A, B \in \Sigma$. Then $i^\infty$ satisfies:\\
$\mathrm{(1)}$ $A  = \emptyset \Rightarrow  i^\infty = -\infty$;\\
$\mathrm{(2)}$ (Monotonicity) $A \subset B \Rightarrow i^\infty(A)
\leq i^\infty(B)$;\\
$\mathrm{(3)}$ (Subadditivity) $i^\infty(A\cup B) \leq i^\infty(A) + i^\infty(B)$;\\
$\mathrm{(4)}$ If $V\cap \mathrm{Fix}(G) = \{0\}$, then $i^\infty(S_\rho\cap V) = 0$, where $S_\rho = \{z \in Z: \|z\| = \rho\}$;\\
$\mathrm{(5)}$ If $Y_0$ and $\widetilde{Y_0}$ are $G$-invariant
closed subspaces of $V$ such that $V = Y_0 \oplus \widetilde{Y_0}$,
$\widetilde{Y_0} \subset V_{j_0}$ for some $j_0$ and
$\mathrm{dim}(Y_0) = dm$, then $i^\infty(S_\rho\cap Y_0) \geq -m$.
\end{proposition}

\begin{definition}(see \cite{w1})\label{de2.5}
A functional $I \in C^1(Z, R)$ is said to satisfy the condition
$(PS)_c^\ast$ if any sequence $\{u_{n_k}\}$, $u_{n_k} \in Z_{n_k}$
such that
\begin{displaymath}
I(u_{n_k}) \rightarrow c,\quad dI_{n_k}(u_{nk}) \rightarrow 0, \quad
\mbox{as}\ k \rightarrow \infty,
\end{displaymath}
possesses a convergent subsequence, where $Z_{n_k}$ is the
$n_k$-dimensional subspace of $Z$, $I_{n_k} = I|_{Z_{n_k}}$.
\end{definition}

\begin{theorem}(see \cite{Li1})\label{the3.1}
Assume that\\
$\mathrm{(D_1)}$ $I \in C^1(Z, R)$ is $G$-invariant;\\
$\mathrm{(D_2)}$ There are $G$-invariant closed subspaces $U$ and $V$ such that $V$ is infinite-dimensional and $Z = U\oplus V$;\\
$\mathrm{(D_3)}$ There is a sequence of $G$-invariant finite-dimensional subspaces
\begin{displaymath}
V_1 \subset V_2 \subset \cdots \subset V_j \subset \cdots, \quad
\mathrm{dim}(V_j) = dn_j,
\end{displaymath}
such that $V = \overline{\cup_{j=1}^\infty V_j}$;\\
$\mathrm{(D_4)}$ There is an index theory $i$ on $Z$ satisfying the $d$-dimensional property;\\
$\mathrm{(D_5)}$  There are $G$-invariant subspaces $Y_0$,
$\widetilde{Y_0}$, $Y_1$ of $V$ such that $V = Y_0 \oplus
\widetilde{Y_0}$, $Y_1, \widetilde{Y_0} \subset V_{j_0}$ for some
$j_0$ and $\mathrm{dim}(\widetilde{Y_0}) = dm < dk = \mathrm{dim}(Y_1)$;\\
$\mathrm{(D_6)}$  There are $\alpha$ and $\beta$, $\alpha < \beta,$
such that $f$ satisfies $(PS)_c^\ast$, $ \ \hbox{for all} \ \ c \in [\alpha,
\beta]$; \\
\begin{displaymath}
\hskip -1cm \mathrm{(D_7)}\quad \left\{\begin{array}{lll}(a)\
either\ \mathrm{Fix}(G) \subset U \oplus
Y_1, \quad or \quad  \mathrm{Fix}(G) \cap V = \{0\},\\
(b)\ there\ is\ \rho > 0\ such\ that\  \ \hbox{for all} \ \ u \in Y_0 \cap
S_\rho, f(z) \geq \alpha,\\
(c)\  \ \hbox{for all} \ \ z \in U \oplus Y_1,\ f(z) \leq \beta,
\end{array}\right.
\end{displaymath}
if $i^\infty$ is the limit index corresponding to $i$, then the
numbers
\begin{displaymath}
c_j = \inf_{i^\infty(A)\geq j}\sup_{z \in A}f(u),\quad -k+1 \leq j
\leq -m,
\end{displaymath}
are critical values of $f$, and $\alpha \leq c_{-k+1} \leq \cdots
\leq c_{-m} \leq \beta$. Moreover, if $c = c_l = \cdots = c_{l+r}$,
$r \geq 0$, then $i(\mathbb{K}_c) \geq r + 1$, where $\mathbb{K}_c =
\{z \in Z: df(z) = 0, f(z)=c\}$.
\end{theorem}

\section{Proof of $(PS)_c$ condition}\label{s4}

In this section, to overcome the lack of compactness caused by the
critical exponents, we intend to employ the second
concentration-compactness principle introduced in Li et al.
\cite{li2}. In consideration of the appearance of convolution, it is
natural to consider a variant of the concentration-compactness
principle for the convolution type problem in the fractional Sobolev
space. Since the proof is similar to that of \cite{gao1, lions1,
zhang2}, we omit the details.

\begin{lemma}\label{lem4.1}(see \cite{li2})
Assume that $\{u_{n}\}$ be a bounded sequence in $H^s(\mathbb{R}^N)$
satisfying $u_n \rightharpoonup u$ weakly in $H^s(\mathbb{R}^N)$,
$u_n \rightarrow u$ strongly in $L^2_{loc}(\mathbb{R}^N)$ and $u_n
\rightarrow u$ a.e. on $\mathbb{R}^N$. Let $|(-\Delta)^{\frac{s}{2}}
u_n|^2 \rightharpoonup \omega$, $|u_n|^{2_s^\ast}\rightharpoonup
\xi$ and
$\left(\mathcal{I}_\mu*|u_n|^{p^*}\right)|u_n|^{p^\ast}\rightharpoonup
\nu$ weakly in the sense of measures, where $\omega$, $\xi$  and
$\nu$ are bounded nonnegative measures on $\mathbb{R}^N$. Define
\begin{eqnarray*}
\omega_\infty = \lim\limits_{R\rightarrow
\infty}\limsup\limits_{n\rightarrow\infty}\displaystyle\int_{\{x \in
\mathbb{R}^N: |x|>R\}}|(-\Delta)^{\frac{s}{2}}
u_n|^{2}dx,
\end{eqnarray*}
\begin{eqnarray*}
\xi_\infty = \lim\limits_{R\rightarrow
\infty}\limsup\limits_{n\rightarrow\infty}\displaystyle\int_{\{x \in
\mathbb{R}^N: |x|>R\}}|u_n|^{2_s^\ast}dx,
\end{eqnarray*} and
\begin{eqnarray*}
\nu_\infty = \lim\limits_{R\rightarrow
\infty}\limsup\limits_{n\rightarrow\infty}\displaystyle\int_{\{x \in
\mathbb{R}^N:
|x|>R\}}\left(\mathcal{I}_\mu*|u_n|^{p^*}\right)|u_n|^{p^\ast}dx.\end{eqnarray*}
Then there exists an (at most countable) set of distinct points
$\{x_i\}_{i\in I} \subset \mathbb{R}^N$ and a family of positive
numbers $\{\nu_i\}_{i\in I}$ such that
\begin{eqnarray}\label{e4.1}
\nu = \left(\mathcal{I}_\mu*|u|^{p^*}\right)|u|^{p^\ast} + \sum_{i
\in I} \nu_i\delta_{x_i},\quad \sum_{i \in I}
\nu_i^{\frac{N}{N+\mu}} < \infty,
\end{eqnarray}
\begin{eqnarray}\label{e4.2}
\xi \geq |u|^{2_s^\ast} + C_\mu(N) ^{-\frac{N}{N+\mu}}\sum_{i \in I}
\nu_i^{\frac{N}{N+\mu}}\delta_{x_i},\quad \xi_i \geq C_\mu(N)
^{-\frac{N}{N+\mu}}\nu_i^{\frac{N}{N+\mu}},
\end{eqnarray}
and
\begin{eqnarray}\label{e4.3}
\omega \geq |(-\Delta)^{\frac{s}{2}} u|^2 + S_H\sum_{i \in I}
\nu_i^{\frac{1}{p^\ast}}\delta_{x_i}, \quad  \omega_i \geq
 S_H\nu_i^{\frac{1}{p^\ast}},
\end{eqnarray}
where $\delta_{x_i}$ is the Dirac-mass of mass 1 concentrated at $x
\in \mathbb{R}^N$. For the energy at infinity, we have
\begin{eqnarray}\label{e4.4}
\limsup\limits\limits_{n\rightarrow\infty}\displaystyle\int_{\mathbb{R}^N}\left(\mathcal{I}_\mu*|u_n|^{p^*}\right)|u_n|^{p^\ast}dx
= \displaystyle\int_{\mathbb{R}^N}d\nu + \nu_\infty,
\end{eqnarray}
\begin{eqnarray}\label{e4.5}
\limsup\limits_{n\rightarrow\infty}\displaystyle\int_{\mathbb{R}^N}|(-\Delta)^{\frac{s}{2}}
u_n|^{2}dx = \int_{\mathbb{R}^N}d\omega + \omega_\infty,
\end{eqnarray}
\begin{eqnarray}\label{e4.6}
\limsup\limits\limits_{n\rightarrow\infty}\displaystyle\int_{\mathbb{R}^N}|u_n|^{2_s^\ast}dx
= \displaystyle\int_{\mathbb{R}^N}d\xi + \xi_\infty,
\end{eqnarray}
\begin{eqnarray}\label{e4.7}
\xi_\infty \leq
\left(S^{-1}\omega_\infty\right)^{\frac{2_s^\ast}{2}},
\end{eqnarray}
\begin{eqnarray}\label{e4.8}
\nu_\infty \leq C_\mu(N)\left(\int_{\mathbb{R}^N}d\xi +
\xi_\infty\right)^{\frac{N+\mu}{2N}}\xi_\infty^{\frac{N+\mu}{2N}},
\end{eqnarray}
and
\begin{eqnarray}\label{e4.9}
\nu_\infty \leq S_H^{-p^\ast}\left(\int_{\mathbb{R}^N}d\omega +
\omega_\infty\right)^{\frac{p^\ast}{2}}\omega_\infty^{\frac{p^\ast}{2}}.
\end{eqnarray}
\end{lemma}

Now, in order to show that all hypotheses of the Limit Index Theory
are satisfied for  \ref{the3.1}, we denote  $G_1 = O(N)$, where $O(N)$ is
the group of orthogonal linear transformations in $\mathbb{R}^N$, $E
= H_{r,A}^{s}(\mathbb{R}^N,\mathbb{C})$ and $$E_{G_1} =
H_{r,A,O(N)}^{s} := \{u\in H_{r,A}^{s}(\mathbb{R}^N,\mathbb{C}):
gu(x) = u(g^{-1}x) = u(x), g \in O(N)\}.$$   For convenience, let
$G_2 = \mathbb{Z}_2$, $Y = E \times E$, $X = Y_{G_1} = E_{G_1}\times
E_{G_1}$. The space $Y$ is endowed with the norm $\|(u, v)\|_{s,A} =
\|u\|_{s,A} + \|v\|_{s,A}$. Using d'Avenia and  Squassina \cite{da},
it is easy to prove that $(Y, \|\cdot\|_{s,A})$ is a reflexive
Banach space. The corresponding energy functional of problem
\eqref{e1.1} is given by
\begin{eqnarray}\label{e4.10}
\mathcal {J}(u,v) &=&\nonumber
 -\frac{1}{2}\mathscr{M}(\|u\|_{s,A}^2) + \frac{1}{2}\mathscr{M}(\|v\|_{s,A}^2) -\frac{1}{2{p^*}}\int_{\mathbb
R^{N}}\left(\mathcal{I}_\mu*|u|^{p^*}\right)|u|^{p^\ast}\,dx\\
&& -\frac{1}{2{p^*}}\int_{\mathbb
R^{N}}\left(\mathcal{I}_\mu*|v|^{p^*}\right)|v|^{p^\ast}\,dx -
\frac{1}{2}\int_{\mathbb{R}^N}G(|x|, |u|^2,|v|^2)dx
\end{eqnarray}
for each $(u,v)\in Y$. By condition ($\mathcal {F}$), we know that
the functional $\mathcal {J}$ is well defined on $Y$ and belongs to
$C^1(Y, \mathbb{R})$. Moreover,  its Fr\'{e}chet derivative is given
by
\begin{eqnarray*}
\langle \mathcal {J}'(u,v), (z_1,z_2)\rangle &=& \nonumber
-\mathfrak{M}(\|u\|_{s,A}^2)\left(\langle
u,z_1\rangle_{s,A}+\mathscr{R}\int_{\mathbb{R}^N}u\overline{z}_1dx\right)
+ \mathfrak{M}(\|v\|_{s,A}^2)\left(\langle
v,z_2\rangle_{s,A}+\mathscr{R}\int_{\mathbb{R}^N}v\overline{z}_2dx\right)  \\
&&\nonumber -\mathscr{R}\int_{\mathbb{R}^N}(\mathcal{I}_\mu*|u|^{p^*})|u|^{p^*-2}u\overline{z}_1 dx-\mathscr{R}\int_{\mathbb{R}^N}(\mathcal{I}_\mu*|v|^{p^*})|v|^{p^*-2}v\overline{z}_2 dx\\
&&   - \mathscr{R}\int_{\mathbb{R}^N}F_u(|x|,
|u|^2,|v|^2)u\overline{z}_1dx -
\mathscr{R}\int_{\mathbb{R}^N}F_v(|x|, |u|^2,|v|^2)v\overline{z}_2dx
= 0
\end{eqnarray*}
for any $(u,v), (z_1,z_2) \in Y$, where
$$\langle \zeta, z_i\rangle_{s,A} =
\mathscr{R}\iint_{\mathbb{R}^{2N}}\frac{(\zeta(x)-e^{i(x-y)\cdot
A(\frac{x+y}{2})}\zeta(y))\overline{(z_i(x)-e^{i(x-y)\cdot
A(\frac{x+y}{2})}z_i(y))}}{|x-y|^{N+2s}}dx dy$$ for any $\zeta, z_i
\in H_A^s(\mathbb{R}^N, \mathbb{C})$ $(i = 1,2)$.  By condition
($\mathcal {F}$),  we know that the functional $\mathcal {J}$ is
$O(N)$-invariant. Therefore, from  the principle of symmetric
criticality of Krawcewicz and Marzantowicz \cite{kr}, we know that
$(u, v)$ is a critical point of $\mathcal {J}$ if and only if
$(u,v)$ is a critical point of $J = \mathcal {J}|_{X= E_{G_1}\times
E_{G_1}}$.  So, we just need to prove 
 the existence of a
sequence of critical points $\mathcal {J}$ on $Y$.

Now, we begin to prove the $(PS)_c$ condition.
\begin{lemma}\label{lem4.2}
Let ($\mathcal {M}$) and ($\mathcal {G}$) hold,
$\{(u_{n_k},v_{n_k})\}$ be a sequence such that
$\{(u_{n_k},v_{n_k})\} \in X_{n_k}$,
\begin{displaymath}
J_{n_k}(u_{n_k},v_{n_k}) \rightarrow c < c^\ast\quad \mbox{as}\ k
\rightarrow \infty,
\end{displaymath}
where $J_{n_k} = \mathcal {J}|_{X_{n_k}}$ and
$$c^\ast = \min\left\{\left(\frac{1}{\theta}-\frac{1}{2{p^*}}\right)\left(\mathfrak{m}_1S_H^\sigma\right)^{\frac{p^\ast}{p^\ast-\sigma}},
\left(\frac{1}{\theta}-\frac{1}{2{p^*}}\right)\left(\mathfrak{m}_1\hat{C}(\mu,N)^{-1}S^{\frac{p^\ast}{2}}\right)^{\frac{2}{p^\ast-2\sigma}}\right\}.$$
Then  there exists a subsequence of $\{(u_{n_k},v_{n_k})\}$ strongly
convergent in $X$.
\end{lemma}
\begin{proof}[\bf Proof]
If $\inf_{n\in\mathbb N}\|u_{n_k}\|_{s,A}= 0$ or $\inf_{n\in\mathbb
N}\|v_{n_k}\|_{s,A}= 0$, then there exists a subsequence of
$\{u_{n_k}\}$ (or $\{v_{n_k}\}$)  such that $u_{n_k} \rightarrow 0$
or $v_{n_k} \rightarrow 0$ in $E_{G_1}$ as $n \rightarrow \infty$.
Thus, we can
 assume that  $\inf_{n\in\mathbb N}\|u_{n_k}\|_{s}=d_1>0$
and $\inf_{n\in\mathbb N}\|v_{n_k}\|_{s}=d_2>0$ in the
sequel.

On the one hand, from   $(M_1)$ and $(F_3)$, we get
\begin{eqnarray}\label{e4.11}
o(1)\|u_{n_k}\|_{s,A} &\geq& \nonumber \langle -dJ_{n_k}(u_{n_k}, v_{n_k}), (u_{n_k}, 0) \rangle\\
&=& \nonumber \mathfrak{M}(\|u_{n_k}\|_{s,A}^2)\|u_{n_k}\|_{s,A}^2
+ \int_{\mathbb
R^{N}}\left(\mathcal{I}_\mu*|u_{n_k}|^{p^*}\right)|u_{n_k}|^{p^\ast}\,dx\\
&& \nonumber  + \int_{\mathbb{R}^N}G_u(|x|,
|u_{n_k}|^2, |v_{n_k}|^2)|u_{n_k}|^2dx 
\geq \mathfrak{m}^\ast\|u_{n_k}\|_{s,A}^2.
\end{eqnarray}
Therefore,$\{u_{n_k}\}$ is bounded in $E_{G_1}$. On the other hand,
from $(F_2)$  and the fact that $2\sigma <\theta<2p^*$, we have
\begin{eqnarray}\label{e4.12}
c + o(1)\|v_{n_k}\|_{s,A} &=&\nonumber J_{n_k}(0, v_{n_k}) -  \frac{1}{\theta}\langle dJ_{n_k}(u_{n_k}, v_{n_k}), (0, v_{n_k}) \rangle \\
&=&\nonumber \frac{1}{2}\mathscr{M}(\|v_{n_k}\|_{s,A}^2)-
 \frac{1}{\theta}\mathfrak{M}(\|v_{n_k}\|_{s,A}^2)\|v_{n_k}\|_{s,A}^2 \\
 &&\nonumber
+\left(\frac{1}{\theta}-\frac{1}{2{p^*}}\right)\int_{\mathbb
R^{N}}\left(\mathcal{I}_\mu*|v_{n_k}|^{p^*}\right)|v_{n_k}|^{p^\ast}\,dx\\
&&\nonumber - \int_{\mathbb{R}^N} \left[G(|x|, 0,|v_{n_k}|^2) -
\frac{1}{\theta}G_v(|x|,
0,|v_{n_k}|^2)|v_{n_k}|^2\right]dx\\
&\geq&%\nonumber 
\left(\frac{1}{2\sigma}-\frac{1}{\theta}\right)\mathfrak{M}(\|v_{n_k}\|_{s,A}^2)\|v_{n_k}\|_{s,A}^2
\geq \left(\frac{1}{2\sigma}-\frac{1}{\theta}\right)\mathfrak{m}^\ast\|v_{n_k}\|_{s,A}^2.
\end{eqnarray}
This fact  implies that  $\{v_{n_k}\}$ is bounded in $E_{G_1}$. Thus
$\|u_{n_k}\|_{s,A} + \|v_{n_k}\|_{s,A}$ is bounded in $X$.

\indent Now, we shall prove that  $\{(u_{n_k},v_{n_k})\}$
contains a subsequence strongly convergent in $X$.
Since $\{u_{n_k}\}$ is bounded in $E_{G_1}$, it follows that, up to a
subsequence, $u_{n_k} \rightharpoonup u_0$ weakly in $E_{G_1}$ and
$u_{n_k} \rightarrow u_0$, a.e. in $\mathbb{R}^N$. Thus, it follows
from $(M_1)$ and $(F_3)$ that
\begin{eqnarray*}
0 &\leftarrow & \langle -dJ_{n_k}(u_{n_k}-u_0, v_{n_k}), (u_{n_k}-u_0, 0)\rangle\\
&=& \mathfrak{M}(\|u_{n_k}-u_0\|_{s,A}^2)\|u_{n_k}-u_0\|_{s,A}^2 +
 \int_{\mathbb{R}^N}G_u(|x|,
|u_{n_k}-u_0|^2, |v_{n_k}|^2)|u_{n_k}-u_0|^2dx \\
&& \mbox{} + \int_{\mathbb
R^{N}}\left(\mathcal{I}_\mu*|u_{n_k}-u_0|^{p^*}\right)|u_{n_k}-u_0|^{p^\ast}\,dx
\geq  \mathfrak{m}^\ast\|u_{n_k}-u_0\|_{s,A}^2,
\end{eqnarray*}
which implies that
\begin{align}\label{e4.13}
u_{n_k} \rightarrow u_0 \quad \mbox{ strongly\ in\ } E_{G_1}.
\end{align}
It suffices to prove that there exists $v_0 \in E_{G_1}$ such that
\begin{align}\label{e4.14}
v_{n_k} \rightarrow v_0 \quad \mbox{ strongly\ in\ } E_{G_1}.
\end{align}

Next, in order to prove \eqref{e4.14}, we divide the following proof
into three claims.

\vspace{2mm}

{\bf Claim 1.} Fix $i \in I$. Then  either  $\nu_i = 0$
or
\begin{eqnarray}\label{e4.15}
\nu_i \geq
\left(\mathfrak{m}_1S_H^\sigma\right)^{\frac{p^\ast}{p^\ast-\sigma}}.
\end{eqnarray}

In order to prove \eqref{e4.15}, we take $\phi \in
C_0^\infty(\mathbb{R}^N)$ be a radial symmetric function satisfying
$0 \leq \phi \leq 1$; $\phi \equiv 1$ in $B(x_i, \epsilon)$,
$\phi(x) = 0$ in $\mathbb{R}^N \setminus B(x_i, 2\epsilon)$. For any
$\epsilon
> 0$, define $\phi_\epsilon :=
\phi\left(\frac{x-x_i}{\epsilon}\right)$, where $i \in I$. Clearly
$\{\phi_\epsilon v_{n_k}\}$ is bounded in $H_{r,A}^s(\mathbb{R}^N,
\mathbb{C})$ and $\langle dJ_{n_k}(u_{n_k},v_{n_k}), (0,
v_{n_k}\phi_\epsilon)\rangle \rightarrow 0$  as $n \rightarrow
\infty$. Hence
\begin{eqnarray}\label{e4.16}
&&\nonumber
\mathfrak{M}\left(\|v_{n_k}\|_{s,A}^2\right)\left(\iint_{\mathbb{R}^{2N}}\frac{|v_{n_k}(x)-e^{i(x-y)\cdot
A(\frac{x+y}{2})}v_{n_k}(y)|^2\phi_\epsilon(y)}{|x-y|^{N+2s}}dx dy +
\int_{\mathbb{R}^N} |v_{n_k}|^2\phi_\epsilon(x) dx\right)\\
&&\mbox{}\nonumber = - \mathscr{R}\left\{
\mathfrak{M}\left(\|v_{n_k}\|_{s,A}^2\right)\iint_{\mathbb{R}^{2N}}\frac{(v_{n_k}(x)-e^{i(x-y)\cdot
A(\frac{x+y}{2})}v_{n_k}(y))\overline{v_{n_k}(x)(\phi_\epsilon(x)-\phi_\epsilon(y))}}{|x-y|^{N+2s}}dx dy\right\}\\
&&\mbox{}\ \  + \int_{\mathbb
R^{N}}\left(\mathcal{I}_\mu*|v_{n_k}|^{p^*}\right)|v_{n_k}|^{p^\ast}\phi_\epsilon\,dx
 + \int_{\mathbb{R}^N}G_v(|x|,
|u_{n_k}|^2, |v_{n_k}|^2)|v_{n_k}|^2\phi_\epsilon dx +o_n(1).
\end{eqnarray}
We deduce from $(M_2)$ and diamagnetic inequality  that
\begin{eqnarray}\label{e4.17}
&&\nonumber
\mathfrak{M}\left(\|v_{n_k}\|_{s,A}^2\right)\left(\iint_{\mathbb{R}^{2N}}\frac{|v_{n_k}(x)-e^{i(x-y)\cdot
A(\frac{x+y}{2})}v_{n_k}(y)|^2\phi_\epsilon(y)}{|x-y|^{N+2s}}dx dy +
\int_{\mathbb{R}^N} |v_{n_k}|^2\phi_\epsilon(x) dx\right) \\
&&\nonumber \mbox{} \ \ \geq \mathfrak{m}_1
\left(\iint_{\mathbb{R}^{2N}}\frac{|v_{n_k}(x)-e^{i(x-y)\cdot
A(\frac{x+y}{2})}v_{n_k}(y)|^2\phi_\epsilon(y)}{|x-y|^{N+2s}}dx dy +
\int_{\mathbb{R}^N} |v_{n_k}|^2\phi_\epsilon(x) dx\right)^{\sigma}\\
&&\geq \mathfrak{m}_1
\left(\iint_{\mathbb{R}^{2N}}\frac{\left||v_{n_k}(x)|-|v_{n_k}(y)|\right|^2\phi_\epsilon(y)}{|x-y|^{N+2s}}dxdy\right)^{\sigma}.
\end{eqnarray}
We note that
\begin{eqnarray}\label{e4.18}
\lim_{\epsilon \rightarrow 0}\lim_{n \rightarrow
\infty}\iint_{\mathbb{R}^{2N}}\frac{\left||v_{n_k}(x)|-|v_{n_k}(y)|\right|^2\phi_\epsilon(y)}{|x-y|^{N+2s}}dxdy
= \lim_{\epsilon \rightarrow 0}\int_{\mathbb{R}^{N}}\phi_\epsilon
d\omega = \omega_i
\end{eqnarray}
and
\begin{eqnarray}\label{e4.19}
\lim_{\epsilon \rightarrow 0}\lim_{n \rightarrow \infty}
\int_{\mathbb
R^{N}}\left(\mathcal{I}_\mu*|v_{n_k}|^{p^*}\right)|v_{n_k}|^{p^\ast}\phi_\epsilon\,dx
= \lim_{\epsilon \rightarrow 0}\int_{\mathbb{R}^{N}}\phi_\epsilon
d\nu = \nu_i.
\end{eqnarray}
By the  H\"{o}lder inequality, we have
\begin{eqnarray}\label{e4.20}
&&\nonumber\left|\mathscr{R}\left\{
\mathfrak{M}\left(\|v_{n_k}\|_{s,A}^2\right)\int\int_{\mathbb{R}^{2N}}\frac{(v_{n_k}(x)-e^{i(x-y)\cdot
A(\frac{x+y}{2})}v_{n_k}(y))\overline{v_{n_k}(x)(\phi_\epsilon(x)-\phi_\epsilon(y))}}{|x-y|^{N+2s}}dx dy\right\}\right|\\
&&\nonumber \mbox{} \ \leq
C\iint_{\mathbb{R}^{2N}}\frac{|v_{n_k}(x)-e^{i(x-y)\cdot
A(\frac{x+y}{2})}v_{n_k}(y)|\cdot|\phi_\epsilon(x)-\phi_\epsilon(y)|\cdot|v_{n_k}(x)|}{|x-y|^{N+2s}}dxdy
\\
&& \mbox{} \  \leq C
\left(\iint_{\mathbb{R}^{2N}}\frac{|v_{n_k}(x)|^2|\phi_\epsilon(x)-\phi_\epsilon(y)|^2}{|x-y|^{N+2s}}dxdy\right)^{1/2}.
\end{eqnarray}
\indent  As the proof of  Zhang et al. \cite[Lemma 3.4]{zhang2}, we
get
\begin{eqnarray}\label{e4.21}
\lim_{\epsilon\rightarrow
 0}\lim_{n\rightarrow\infty}\iint_{\mathbb{R}^{2N}}\frac{|v_{n_k}(x)|^2|\phi_\epsilon(x)-\phi_\epsilon(y)|^2}{|x-y|^{N+2s}}dxdy
=0.
\end{eqnarray}
Furthermore, the Lebesgue dominated convergence theorem and $(F_1)$
imply that
\begin{eqnarray}\label{e4.22}
\int_{\mathbb{R}^{N}}G_v(|x|,
|u_{n_k}|^2,|v_{n_k}|^2)|v_{n_k}|^2\phi_{\epsilon}(x) \mathrm{d}x
\rightarrow
  \int_{\mathbb{R}^{N}}G_v(|x|,
|u_{0}|^2,|v_{0}|^2)|v_{0}|^2\phi_{\epsilon}(x) \mathrm{d}x \quad
\mbox{as}\ n\to \infty.
\end{eqnarray}
The definition of $\phi_{\epsilon}(x)$ gives us
\begin{eqnarray}\label{e4.23}
\left|\int_{\mathbb{R}^{N}}G_v(|x|,
|u_{0}|^2,|v_{0}|^2)|v_{0}|^2\phi_{\epsilon}\mathrm{d}x\right| \leq
\int_{B_\epsilon(0)}|G_v(|x|,
|u_{0}|^2,|v_{0}|^2)|v_{0}|^2\mathrm{d}x \rightarrow 0 \quad
\mbox{as}\ \epsilon \to 0.
\end{eqnarray}
Combining \eqref{e4.16}--\eqref{e4.21}, we get that $\nu_i \geq
\mathfrak{m}_1\omega_i^\sigma.$ It follows from \eqref{e4.3}  that
$\nu_i = 0$ or
\begin{eqnarray*}
\nu_i \geq
\left(\mathfrak{m}_1S_H^\sigma\right)^{\frac{p^\ast}{p^\ast-\sigma}}.
\end{eqnarray*}

{\bf Claim 2.}  $\nu_i = 0$, $ \ \hbox{for all} \ \ i \in I$ and
$\nu_\infty = 0$.\\

Indeed, if Claim 2 were false,  then there would exist a $i\in I$ such
that \eqref{e4.15} would hold.  Similar to \eqref{e4.12}, by $(M_3)$ and
$(F_2)$, we deduce
\begin{align}\label{e4.24}
  c &=\nonumber\lim_{\epsilon\rightarrow 0}  \lim_{ n \rightarrow
\infty} \left(J_{n_k}(0, v_{n_k}) -  \frac{1}{\theta}\langle
dJ_{n_k}(u_{n_k}, v_{n_k}), (0, v_{n_k}) \rangle\right)\\
&\geq\nonumber
\left(\frac{1}{\theta}-\frac{1}{2{p^*}}\right)\int_{\mathbb
R^{N}}\left(\mathcal{I}_\mu*|v_{n_k}|^{p^*}\right)|v_{n_k}|^{p^\ast}\,dx\\
&\geq\nonumber
\left(\frac{1}{\theta}-\frac{1}{2{p^*}}\right)\int_{\mathbb
R^{N}}\left(\mathcal{I}_\mu*|v_{n_k}|^{p^*}\right)|v_{n_k}|^{p^\ast}\phi_\epsilon\,dx
\\
&\geq\left(\frac{1}{\theta}-\frac{1}{2{p^*}}\right)\nu_i \geq
\left(\frac{1}{\theta}-\frac{1}{2{p^*}}\right)\left(\mathfrak{m}_1S_H^\sigma\right)^{\frac{p^\ast}{p^\ast-\sigma}}.
\end{align}

On the other hand, we show that $\nu_\infty = 0$. For this, we take
a cut off function $\phi_R \in C^\infty(\mathbb{R}^N)$ such that
\begin{equation*}
\phi_R(x) = \begin{cases}
    0 &\quad |x| < R,\\
    1 &\quad |x| > R+1
    \end{cases}
\end{equation*}
and $|\nabla\phi_R| \leq 2/R$. By using the Hardy-Littlewood-Sobolev
and H\"older's inequality, we get
\begin{align}\label{e4.25}
\nu_\infty &=\nonumber   \lim_{R\rightarrow \infty}  \lim_{ n
\rightarrow
\infty}\displaystyle\int_{\mathbb{R}^N}\left(\mathcal{I}_\mu*|v_{n_k}|^{p^*}\right)|v_{n_k}|^{p^\ast}\phi_R(y)dx  \\
&
%\nonumber 
\leq C_\mu(N)\lim_{R\rightarrow \infty}  \lim_{ n
\rightarrow
\infty}|v_{n_k}|_{2_s^\ast}^{p^\ast}\left(\int_{\mathbb{R}^N}|v_{n_k}(x)|^{2_s^*}\phi_R(y)dx\right)^{\frac{p^\ast}{2_s^\ast}}
\leq
\hat{C}(\mu,N)\xi_\infty^{\frac{p^\ast}{2_s^\ast}}.\end{align}

Note that $\{\phi_R v_{n_k}\}$ is bounded in
$H_{r,A}^s(\mathbb{R}^N, \mathbb{C})$. Hence, $\langle
dJ_{n_k}(u_{n_k},v_{n_k}), (0, v_{n_k}\phi_R)\rangle \rightarrow 0$
as $n \rightarrow \infty$, which yields that
\begin{eqnarray}\label{e4.26}
&&\nonumber
\mathfrak{M}\left(\|v_{n_k}\|_{s,A}^2\right)\left(\iint_{\mathbb{R}^{2N}}\frac{|v_{n_k}(x)-e^{i(x-y)\cdot
A(\frac{x+y}{2})}v_{n_k}(y)|^2\phi_R(y)}{|x-y|^{N+2s}}dx dy +
\int_{\mathbb{R}^N} |v_{n_k}|^2\phi_R(x) dx\right)\\
&&\mbox{}\nonumber = - \mathscr{R}\left\{
\mathfrak{M}\left(\|v_{n_k}\|_{s,A}^2\right)\iint_{\mathbb{R}^{2N}}\frac{(v_{n_k}(x)-e^{i(x-y)\cdot
A(\frac{x+y}{2})}v_{n_k}(y))\overline{v_{n_k}(x)(\phi_R(x)-\phi_R(y))}}{|x-y|^{N+2s}}dx dy\right\}\\
&&\mbox{}\ \  + \int_{\mathbb
R^{N}}\left(\mathcal{I}_\mu*|v_{n_k}|^{p^*}\right)|v_{n_k}|^{p^\ast}\phi_R\,dx
+
  \int_{\mathbb{R}^N}G_v(|x|,
|u_{n_k}|^2, |v_{n_k}|^2)|v_{n_k}|^2\phi_R dx +o_n(1).
\end{eqnarray}
We can easily deduce that
\begin{eqnarray*}
\limsup\limits_{R\rightarrow\infty}\limsup\limits_{n\rightarrow\infty}\iint_{\mathbb{R}^{2N}}\frac{||v_{n_k}(x)|-|v_{n_k}(y)||^2\phi_R(y)}{|x-y|^{N+2s}}dxdy
= \omega_\infty
\end{eqnarray*}
and
\begin{eqnarray*}
&& \left|\mathscr{R}\left\{
\mathfrak{M}\left(\|v_{n_k}\|_{s,A}^2\right)\iint_{\mathbb{R}^{2N}}\frac{(v_{n_k}(x)-e^{i(x-y)\cdot
A(\frac{x+y}{2})}v_{n_k}(y))\overline{v_{n_k}(x)(\phi_R(x)-\phi_R(y))}}{|x-y|^{N+2s}}dx
dy\right\}\right| \\
&& \mbox{}  \leq
C\left(\iint_{\mathbb{R}^{2N}}\frac{|v_{n_k}(x)|^2|\phi_R(x)-\phi_R(y)|^2}{|x-y|^{N+2s}}dxdy\right)^{1/2}.
\end{eqnarray*}
Furthermore,
\begin{eqnarray*}
&& \limsup\limits_{R \rightarrow \infty}\limsup\limits_{n
\rightarrow \infty}
\iint_{\mathbb{R}^{2N}}\frac{|v_{n_k}(x)|^2|\phi_R(x)-\phi_R(y)|^2}{|x-y|^{N+2s}}dxdy
\\
&& \mbox{} = \limsup\limits_{R \rightarrow \infty}\limsup\limits_{n
\rightarrow \infty}
\iint_{\mathbb{R}^{2N}}\frac{|v_{n_k}(x)|^2|(1-\phi_R(x))-(1-\phi_R(y))|^2}{|x-y|^{N+2s}}dxdy.
\end{eqnarray*}
On the other hand, similar to the proof of  Zhang et al.
\cite[Lemma 3.4 ]{zhang2}, we  have
\begin{eqnarray*}
\limsup\limits_{R \rightarrow \infty}\limsup\limits_{n \rightarrow
\infty}
\iint_{\mathbb{R}^{2N}}\frac{|v_{n_k}(x)|^2|(1-\phi_R(x))-(1-\phi_R(y))|^2}{|x-y|^{N+2s}}dxdy
= 0.
\end{eqnarray*}
It follows from $(M_2)$ that
\begin{eqnarray*}
&&
\mathfrak{M}\left(\|v_{n_k}\|_{s,A}^2\right)\left(\iint_{\mathbb{R}^{2N}}\frac{|v_{n_k}(x)-e^{i(x-y)\cdot
A(\frac{x+y}{2})}v_{n_k}(y)|^2\phi_R(y)}{|x-y|^{N+2s}}dx dy +
\int_{\mathbb{R}^N} |v_{n_k}|^2\phi_R(x) dx\right)\\
&& \mbox{} \ \ \geq  \mathfrak{m}_1
\left(\iint_{\mathbb{R}^{2N}}\frac{|v_{n_k}(x)-e^{i(x-y)\cdot
A(\frac{x+y}{2})}v_{n_k}(y)|^2\phi_R(y)}{|x-y|^{N+2s}}dx dy +
\int_{\mathbb{R}^N} |v_{n_k}|^2\phi_R(x) dx\right)^{\sigma}\\
&& \mbox{} \ \ \geq \mathfrak{m}_1
\left(\iint_{\mathbb{R}^{2N}}\frac{\left||u_n(x)|-|u_n(y)|\right|^2\phi_R(y)}{|x-y|^{N+2s}}dxdy\right)^{\sigma}
= \mathfrak{m}_1 \omega_\infty^{\sigma}.
\end{eqnarray*}
By the definition of $\phi_R$ and conditions $(F_1)$--$(F_2)$, we
have
\begin{eqnarray*}
\lim_{R \rightarrow \infty}\lim_{n \rightarrow
\infty}\int_{\mathbb{R}^N}F_v(|x|, |u_{n_k}|^2,
|v_{n_k}|^2)|v_{n_k}|^2\phi_R dx = \lim_{R \rightarrow
\infty}\int_{\mathbb{R}^N}F_v(|x|, |u_0|^2, |v_0|^2)|v_0|^2\phi_R dx
= 0.
\end{eqnarray*}
Therefore, by \eqref{e4.26} together with \eqref{e4.25}, we can
obtain that
$$\hat{C}(\mu,N)\xi_\infty^{\frac{p^\ast}{2_s^\ast}} \geq  \nu_\infty \geq \mathfrak{m}_1\omega_\infty^\sigma.$$ It follows from \eqref{e4.3}  that $\nu_\infty = 0$ or
\begin{eqnarray*}
\nu_\infty \geq
\left(\mathfrak{m}_1\hat{C}(\mu,N)^{-1}S^{\frac{p^\ast}{2}}\right)^{\frac{2}{p^\ast-2\sigma}}.
\end{eqnarray*}
Then we have
\begin{align}\label{e4.27}
  c &=\nonumber\lim_{R\rightarrow \infty}  \lim_{ n \rightarrow
\infty} \left(J_{n_k}(0, v_{n_k}) -  \frac{1}{\theta}\langle
dJ_{n_k}(u_{n_k}, v_{n_k}), (0, v_{n_k}) \rangle\right)\\
&\geq\nonumber
\left(\frac{1}{\theta}-\frac{1}{2{p^*}}\right)\int_{\mathbb
R^{N}}\left(\mathcal{I}_\mu*|v_{n_k}|^{p^*}\right)|v_{n_k}|^{p^\ast}\phi_R\,dx
\\
&\geq\left(\frac{1}{\theta}-\frac{1}{2{p^*}}\right)\nu_\infty \geq
\left(\frac{1}{\theta}-\frac{1}{2{p^*}}\right)\left(\mathfrak{m}_1\hat{C}(\mu,N)^{-1}S^{\frac{p^\ast}{2}}\right)^{\frac{2}{p^\ast-2\sigma}}.
\end{align}
Invoking the arguments above and \eqref{e4.24},
\eqref{e4.27}, set
$$c^\ast = \min\left\{\left(\frac{1}{\theta}-\frac{1}{2{p^*}}\right)\left(\mathfrak{m}_1S_H^\sigma\right)^{\frac{p^\ast}{p^\ast-\sigma}},
\left(\frac{1}{\theta}-\frac{1}{2{p^*}}\right)\left(\mathfrak{m}_1\hat{C}(\mu,N)^{-1}S^{\frac{p^\ast}{2}}\right)^{\frac{2}{p^\ast-2\sigma}}\right\}.$$
Then for any $c < c^\ast$ we have
$\nu_i =0$ for all $i \in I$ and $ \nu_\infty = 0.$\\

{\bf Claim 3.}  $v_{n_k} \rightarrow v_0$ strongly in $E_{G_1}$. \\

Indeed, by Claim 2 and Lemma 3.1, we know
\begin{eqnarray}\label{e4.28}
\lim_{n \rightarrow \infty} \int_{\mathbb
R^{N}}\left(\mathcal{I}_\mu*|v_{n_k}|^{p^*}\right)|v_{n_k}|^{p^\ast}\,dx
&=\displaystyle\lim_{n \rightarrow \infty}\int_{\mathbb
R^{N}}\left(\mathcal{I}_\mu*|v_0|^{p^*}\right)|v_0|^{p^\ast}\,dx.
\end{eqnarray}

Now, we define the linear functional $\mathcal {L}(v)$ on $E_{G_1}$
as follows
\begin{align*}
[\mathcal {L}(v),w]
&=\mathscr{R}\iint_{\mathbb{R}^{2N}}\frac{(v(x)-e^{i(x-y)\cdot
A(\frac{x+y}{2})}v(y))\overline{(w(x)-e^{i(x-y)\cdot
A(\frac{x+y}{2})}w(y))}}{|x-y|^{N+2s}}dx dy + \mathscr{R}\int_{\mathbb{R}^N}v\overline{w}dx\\
&=\langle
v,w\rangle_{s,A}+\mathscr{R}\int_{\mathbb{R}^N}v\overline{w}dx
\end{align*}
for any $v\in E_{G_1}$.  Obviously, $\mathcal {L}$ is a bounded
bi-linear operator, By the H\"{o}lder inequality, we have
\begin{align*}
\left|[\mathcal {L}(v),w]\right| \leq \|v\|_{s,A}\|w\|_{s,A}.
\end{align*}
Since $v_{n_k}\rightharpoonup v_0$ weakly in $E_{G_1},$ we
have
\begin{align}\label{e4.29}
\lim_{n\rightarrow\infty}[\mathcal {L}(v_0), v_{n_k}-v_0]=0.
\end{align}
Clearly, $[\mathcal {L}(v_{n_k}), v_{n_k}-v_0] \rightarrow 0$ as $n
\rightarrow \infty$. Hence, by \eqref{e4.29}, one has
\begin{align}\label{e4.30}
\lim_{n\rightarrow\infty}[\mathcal {L}(v_{n_k})-\mathcal {L}(v_{0}),
v_{n_k}-v_{0}] =0.
\end{align}
Hence
 by \eqref{e4.29} and \eqref{e4.30}, one has
\begin{align}\label{e4.31}
o(1)&=\langle dJ_{n_k}(u_{n_k}, v_{n_k})-dJ_{n_k}(u_0, v_0),(0,
v_{n_k})-(0, v_0)\rangle
\nonumber\\
&=\mathfrak{M}(\|v_{n_k}\|_{s,A}^2)\|v_{n_k}\|_{s,A}^2-\mathfrak{M}(\|v_{n_k}\|_{s,A}^2)[L(v_{n_k}),v_0]-\mathfrak{M}(\|v_0\|_{s,A}^2)[\mathcal {L}(v_0), v_{n_k}-v_0]\nonumber\\
&\quad -\lambda\int_{\mathbb{R}^N}[G_v(|x|, u_{n_k},
v_{n_k})-G_v(|x|,
u_0, v_0)](v_{n_k}-v_0)dx\nonumber\\
&\quad
-\int_{\mathbb{R}^N}\left[(\mathcal{I}_\mu*|v_{n_k}|^{p^*})|v_{n_k}|^{p^*-2}v_{n_k}
-(\mathcal{I}_\mu*|v_0|^{p^*})|v_0|^{p^*-2}v_0\right]
(v_{n_k}-v_0)dx\\
&=\mathfrak{M}(\|v_{n_k}\|_{s,A}^2)[\mathcal {L}(v_{n_k})-\mathcal
{L}(v_0), v_{n_k}-v_0]
-\int_{\mathbb{R}^N}(\mathcal{I}_\mu*|v_{n_k}-v_0|^{p^*})
|v_{n_k}-v_0|^{p^*}dx+o(1).\nonumber
\end{align}
This fact  together with $\|v_{n_k}\|_{s,A}  \rightarrow \beta$
implies that
\begin{align}\label{e4.32}
\mathfrak{M}(\beta^2)\lim_{n\rightarrow\infty}\|v_{n_k} -
v_0\|_{s,A}^2 = 0.
\end{align}
It follows from $(M_1)$ that  $v_{n_k} \rightarrow v_0$ strongly in
$E_{G_1}$.

To sum up,  we know that $\{(u_{n_k},v_{n_k})\}$ contains a
subsequence converging strongly in $X$ and the proof of Lemma
\ref{lem4.2} is complete.
\end{proof}

\section{Proof of Theorem 1.1}\label{s5}
In this section, we 
 prove that  problem \eqref{e1.1} has
infinitely many solutions.

\vspace{2mm} \noindent{\bf Proof of Theorem \ref{the1.1}} In order
to apply Theorem \ref{the3.1}, we define
\begin{displaymath}
Y = U \oplus V,\quad U = E_{G_1} \times \{0\}, \quad V = \{0\}
\times E_{G_1},
\end{displaymath}
\begin{displaymath}
Y_0 = \{0\} \times E_{G_1}^{m^\perp},\quad Y_1 = \{0\} \times
E_{G_1}^{(k)},
\end{displaymath}
where $m$ and $k$ are
yet
 to be determined. Obviously, $Y_0, Y_1$ are
$G$-invariant and
$$\mbox{codim}_V Y_0 = m, \quad \dim Y_1  =k.$$
It's easy to verify that  $(D_1)$, $(D_2)$, $(D_4)$ in Theorem
\ref{the3.1} are satisfied. Let
$$V_j = E_{G_1}^{(j)} =
\mbox{span}\{e_1,e_2,\cdots,e_j\}.$$ Hence $(D_3)$ in
Theorem~\ref{the3.1} holds.  In order to  verify the conditions
in $(D_7)$ in Theorem~\ref{the3.1}, note that $\mathrm{Fix}(G) \cap
V = 0$, thus
 $(a)$ of $(D_7)$ in Theorem~\ref{the3.1} is satisfied.
Now, we verify that $(b)$ and $(c)$ of $(D_7)$ in
Theorem~\ref{the3.1} holds.

\indent (i) Let $(0, v) \in Y_0\cap S_{\rho_m}$, then  from $(F_1)$ and
$(F_3)$, we get
\begin{eqnarray}\label{e5.1}
J(0, v) &=&\nonumber \frac{1}{2}\mathscr{M}(\|v\|_{s,A}^2)
-\frac{1}{2{p^*}}\int_{\mathbb
R^{N}}\left(\mathcal{I}_\mu*|v|^{p^*}\right)|v|^{p^\ast}\,dx -
\frac{1}{2}\int_{\mathbb{R}^N}G(|x|, 0,|v|^2)dx\\
&\geq& \frac{\mathfrak{m}^\ast}{\sigma}\|v\|_{s,A}^2     -
\frac{S_H^{-p^\ast}}{p_s^\ast}\|v\|_{s,A}^{2p^\ast} -
c\|v\|_{s,A}^{p}.
\end{eqnarray}
Therefore  we can choose $\rho_m > 0$ such that $J(0, v) \geq \alpha$
for $\|v\|_{s,A} = \rho_m$ since $2< p < 2p^\ast$. This fact implies
that $(b)$ of $(D_7)$ in Theorem~\ref{the3.1} holds.

\indent (ii) From $(F_1)$ we have
\begin{eqnarray*}
J(u, 0) &=&\nonumber -\frac{1}{2}\mathscr{M}(\|u\|_{s,A}^2)-
\frac{1}{2{p^*}}\int_{\mathbb
R^{N}}\left(\mathcal{I}_\mu*|u|^{p^*}\right)|u|^{p^\ast}\,dx -
\int_{\mathbb{R}^N}G(|x|, |u|^2, 0)dx\\
&\leq& 0.
\end{eqnarray*}
On the other hand,  we can take $\alpha$ such that
\begin{eqnarray*}
\alpha > \sup_{u \in E_{G_1}}J(u, 0).
\end{eqnarray*}
Let $(u, v) \in U \oplus Y_1$, then  we have
\begin{eqnarray*}
J(u, v) &=&\nonumber -\frac{1}{2}\mathscr{M}(\|u\|_{s,A}^2) +
\frac{1}{2}\mathscr{M}(\|v\|_{s,A}^2) -\frac{1}{2{p^*}}\int_{\mathbb
R^{N}}\left(\mathcal{I}_\mu*|u|^{p^*}\right)|u|^{p^\ast}\,dx\\
&& -\frac{1}{2{p^*}}\int_{\mathbb
R^{N}}\left(\mathcal{I}_\mu*|v|^{p^*}\right)|v|^{p^\ast}\,dx -
\frac{1}{2}\int_{\mathbb{R}^N}G(|x|, |u|^2,|v|^2)dx\\
&\leq&\nonumber \frac{c}{2}\|v\|_{s,A}^2 -
\frac{c}{2p^\ast}\|u\|_{s,A}^{2p^\ast} + \alpha.
\end{eqnarray*}
Note that all norms are equivalent on the finite-dimensional space
$Y_1$, so we can choose $k > m$ and $\beta_k
> \alpha_m$ such that
\begin{align*}
J_{U \oplus Y_1} \leq \beta_k. \end{align*} Hence $(c)$ of $(D_7)$
in Theorem~\ref{the3.1} holds. By Lemma \ref{lem4.2}, $J(u, v)$
satisfies the condition of $(PS)_c^\ast$ for any $c \in [\alpha_m,
\beta_k]$. Therefore
 $(D_6)$ in Theorem \ref{the3.1} holds. Hence by
 Theorem \ref{the3.1}, we know that
\begin{displaymath}
c_j = \inf_{i^\infty(A)\geq j}\sup_{z \in A}J(u,v),\quad -k+1 \leq j
\leq -m,\quad \alpha_m \leq c_j \leq \beta_k,
\end{displaymath}
are critical values of $J$. Letting $m \rightarrow \infty$, we can
obtain an unbounded sequence of critical values $c_j$. Since the
functional $J$ is even, we can get two critical points $(\pm u_j,
\pm
v_j)$ of $J$ corresponding to $c_j$.  \ \ \ $\Box$\\
 
\section*{Acknowledgements}
Shi  was supported by NSFC grant (No. 11771177), China Automobile
Industry Innovation and Development Joint Fund (No. U1664257),
Program for Changbaishan Scholars of Jilin Province and Program for
JLU Science, Technology Innovative Research Team (No. 2017TD-20). 
Repov\v{s} was supported by the Slovenian Research Agency grants (Nos.
P1-0292, N1-0114, N1-0083).


\begin{thebibliography}{99}

\bibitem{am}   V. Ambrosio,  Zero mass case for a fractional Berestycki-Lions-type
problem, Adv. Nonlinear Anal. 7 (2018) 365--374.

\bibitem{benci} {V. Benci, On critical point theory for indefinite functionals in
presence of symmetries, Trans. Amer. Math. Soc. 274 (1982)
533--572.}

\bibitem{cas} D. Cassani, J. Zhang,   Choquard-type equations with
Hardy-Littlewood-Sobolev upper critical growth, Adv. Nonlinear Anal.
8 (2019) 1184--1212.

\bibitem{da1}  P. d'Avenia, G. Siciliano, M. Squassina, On fractional Choquard equations, Math. Models Methods Appl. Sci., 25(8) (2014)  1447--1476.

\bibitem{da}  P. d'Avenia, M. Squassina, Ground states for fractional
magnetic operators, ESAIM: COCV 24 (2018)
1--24.

\bibitem {EGE}  E. Di Nezza, G. Palatucci, E. Valdinoci, Hitchhike's guide to the fractional Sobolev spaces, Bull. Sci. Math.  136
(2012) 521--573.

\bibitem{du} F. D\"{u}zg\"{u}n, Fatma Gamze, A. Iannizzotto,  Three nontrivial
solutions for nonlinear fractional Laplacian equations, Adv.
Nonlinear Anal. 7 (2018) 211--226.

\bibitem{fa} Y. Fang, J. Zhang, Multiplicity of solutions for a class of elliptic systems with
critical Sobolev exponent,  Nonlinear Anal. 73 (2010) 2767--2778.

\bibitem{fis2}  A. Fiscella,  P. Pucci,   B. Zhang, $p$-fractional
Hardy-Schr\"{o}dinger-Kirchhoff systems with critical
nonlinearities, Adv. Nonlinear Anal. 8 (2019)  1111--1131.

\bibitem{fu} M.F. Furtado, L.D. de Oliveira, J.P. da Silva, Multiple solutions
for a critical Kirchhoff system, Appl. Math. Lett. 91 (2019)
97--105.

\bibitem{gao1}  F. Gao, E.D. da Silva, M. Yang, J. Zhou, Existence of solutions for
critical Choquard equations via the concentration compactness
method, Proc. Roy. Soc. Edinburgh Sect. A, doi:10.1017/prm.2018.131.

\bibitem{hu} {D. Huang, Y. Li,  Multiplicity of solutions for a noncooperative
$p$-Laplacian elliptic system in $\mathbb{R}^N$,  J. Differential
Equations 215 (2005) 206--223.}

\bibitem{j1} {C. Ji, V. R\u adulescu, Multiplicity and
concentration of solutions to the nonlinear magnetic Schr\"{o}dinger
equation, Calc. Var. Partial Differential Equations,  59 (2020)  28
pp.}

\bibitem{j2}
{C. Ji, V. R\u adulescu, Multi-bump solutions for the nonlinear
magnetic Schr\"{o}dinger equation with exponential critical growth
in $\mathbb{R}^2$, Manuscripta Math. 164 (2021) 509--542.}


\bibitem{kr}  {W. Krawcewicz, W. Marzantowicz, Some remarks on the Lusternik-Schnirelman
method for non-differentiable functionals invariant with respect to
a finite group action,  Rocky Mountain J. Math. 20 (1990)
1041--1049. }

\bibitem{la} N. Laskin, Fractional quantum mechanics and L\'{e}vy path integrals,  Phys. Lett. A  268 (2000) 298--305.

\bibitem{to} L. Ledesma, E. C\'{e}sar,  Multiplicity result for
non-homogeneous fractional Schr\"{o}dinger-Kirchhoff-type equations
in $\mathbb{R}^N$, Adv. Nonlinear Anal. 7 (2018)  247--257.

\bibitem{Li1} {Y. Li, A limit index theory and its application,  Nonlinear Anal. 25 (1995) 1371--1389. }


\bibitem{li2} { X. Li, S. Ma, G. Zhang, Solutions to upper critical fractional
Choquard equations with potential, Adv. Differential Equations
25(2020) 135--160.}

\bibitem{liang0} S. Liang, G. Molica Bisci, B. Zhang, Multiple solutions for  a noncooperative Kirchhoff-type
system involving the fractional $p$-Laplacian and critical exponents, Math. Nachr. 291 (2018) 1553--1546.



\bibitem{liang3} {S. Liang, D.D. Repov\v{s}, B. Zhang,  On the fractional Schr\"{o}dinger--Kirchhoff equations with electromagnetic
fields and critical nonlinearity, Comput. Math. Appl.  75 (2018) 1778--1794.}

\bibitem{liang4} S. Liang,  D.D. Repov\v{s}, B. Zhang, Fractional magnetic Schr\"{o}dinger-Kirchhoff problems with convolution and critical
nonlinearities, Math. Models Methods Appl. Sci. 43 (2020)
2473--2490.

\bibitem{liang1} S. Liang, S. Shi, Multiplicity of solutions for the
noncooperative $p(x)$-Laplacian operator elliptic system involving
the critical growth, J. Dyn. Control Syst. 18 (2012) 379--396.

\bibitem{liang2}  S. Liang, J. Zhang, Multiple solutions for noncooperative
$p(x)$-Laplacian equations in $\mathbb{R}^N$ involving the critical
exponent, J. Math. Anal. Appl. 403 (2013) 344--356.

\bibitem{lie} E.H. Lieb, M. Loss, Analysis, volume 14 of graduate studies in
mathematics, American Mathematical Society, Providence, RI, (4)
2001.

\bibitem{lions} P. L. Lions, Sym\'{e}trie et compacit\'{e} dans les \'{e}spaces de
Sobolev, J. Funct. Anal. 49 (1982) 315--334.

\bibitem{lions1} {P.L. Lions, The concentration compactness principle in the calculus
of variations. The locally compact case. Part I and II, Ann. Inst.
H. Poincar\'{e} Anal. Non. Lineaire. 1 (1984). pp. 109--145 and
223--283.}

\bibitem{liu} {J. Liu, C. Ji, Concentration results for a magnetic
Schr\"{o}dinger-Poisson system with critical growth,  Adv. Nonlinear
Anal. 10 (2021) 775--798.}

 \bibitem{ma} P. Ma, J. Zhang, Existence and multiplicity of solutions for
fractional Choquard equations, Nonlinear Anal.  164 (2017) 100--117.

\bibitem{MPSZ}  X. Mingqi, P. Pucci, M. Squassina, B. Zhang, Nonlocal Schr$\ddot{\mbox{o}}$dinger--Kirchhoff equations with external magnetic field,
Discrete Contin. Dyn. Syst. 37 (2017) 503--521.

\bibitem{x2}
X. Mingqi, V. R\u adulescu, B. Zhang, A critical fractional
Choquard-Kirchhoff problem with magnetic field, Commun. Contemp.
Math. 21 (2019) 1850004, 36 pp.

\bibitem{x3} X. Mingqi, V. R\u adulescu, B. Zhang, Fractional Kirchhoff
problems with critical Trudinger-Moser nonlinearity, Calc. Var.
Partial Differential Equations 58 (2019) 27 pp.


\bibitem{MRS} G. Molica Bisci, V. R\u adulescu, R. Servadei, Variational methods
for nonlocal fractional equations, Encyclopedia of Mathematics and
its Applications,  162, Cambridge University Press, Cambridge, 2016.

\bibitem{mu1} T. Mukherjee, K. Sreenadh, Fractional Choquard equation with
critical nonlinearities, Nonlinear Differential Equations and
Applications, NoDEA  24(6) (2017) 63.

\bibitem{pa} N. Papageorgiou, V. R\u adulescu, D. Repov\v{s}, Nonlinear
Analysis - Theory and Methods, Springer, Cham, (2019).

\bibitem{pek} S. Pekar, Untersuchung \"{u}ber die Elektronentheorie der Kristalle,
Akademie Verlag, Berlin, 1954.

\bibitem{pu} P. Pucci, M. Xiang, B. Zhang, Existence results for
Schr\"{o}dinger-Choquard-Kirchhoff equations involving the
fractional $p$-Laplacian, Adv. Calc. Vari., 12 (2019) 253--275.

\bibitem{so1} Y. Song,  S. Shi,  Infinitely many solutions for Kirchhoff equations
with Hardy-Littlewood-Sobolev critical nonlinearity, Rev. R. Acad.
Cienc. Exactas F\'{\i}s. Nat. Ser. A Mat. RACSAM  113 (2019)
3223--3232.

\bibitem{so2} Y. Song,  S. Shi,  Existence and multiplicity of solutions for
Kirchhoff equations with Hardy-Littlewood-Sobolev critical
nonlinearity,  Appl. Math. Lett. 92 (2019) 170--175.

\bibitem{so3} Y. Song,  S. Shi, Multiple solutions for a class of noncooperative critical nonlocal
system with variable exponents,  Math. Models Methods Appl. Sci. 44
(2021) 6630--6646.

\bibitem{sq1}  M. Squassina, B. Volzone, Bourgain-Br\'{e}zis-Mironescu formula for
magnetic operators,  C. R. Math. 354 (2016) 825--831.

\bibitem{WX} F. Wang, M. Xiang, Multiplicity of solutions to a nonlocal Choquard equation
involving fractional magnetic operators and critical exponent,
Electron. J. Diff. Equ. 2016 (2016) 1--11.

\bibitem{w1} M. Willem,  Minimax Theorems, Progress in Nonlinear Differential Equations and their Applications, Vol. 24. Birkh\"{a}ser: Boston/Basel/Berlin, 1996.

\bibitem{xia} { A.  Xia,  Multiplicity and concentration results for magnetic
relativistic Schr\"{o}dinger equations,  Adv. Nonlinear Anal. 9
(2020)  1161--1186.}

\bibitem{xiang}  M. Xiang, V. R\u adulescu, B. Zhang,   Combined effects for fractional Schr\"{o}dinger-Kirchhoff systems with
critical nonlinearities,   ESAIM: COCV  24  (2018) 1249--1273.

\bibitem{x1} M. Xiang, B. Zhang, V. R\u adulescu,  Superlinear
Schr\"{o}dinger-Kirchhoff type problems involving the fractional
$p$-Laplacian and critical exponent, Adv. Nonlinear Anal. 9 (2020)
690--709.

\bibitem{zhang2} {X. Zhang, B. Zhang, D.D. Repov\v{s}, Existence and symmetry of
solutions for critical fractional Schr$\ddot{\mbox{o}}$dinger
equations with bounded potentials, Nonlinear Anal. 142 (2016)
48--68.}

\bibitem{zhang3} {Y. Zhang, X. Tang, V. R\u adulescu,  Small perturbations for
nonlinear Schr\"{o}dinger equations with magnetic potential,  Milan
J. Math. 88 (2020)  479--506.}

\end{thebibliography}
\end{document}